\documentclass[11pt]{amsart}
\usepackage{amssymb,amsthm}
\usepackage[nosumlimits]{mathtools}
\usepackage{enumerate}
\usepackage[scr]{rsfso}
\usepackage[margin=1in]{geometry} 
\usepackage{extarrows}

\makeatletter
\newcommand{\subalign}[1]{%
  \vcenter{%
    \Let@ \restore@math@cr \default@tag
    \baselineskip\fontdimen10 \scriptfont\tw@
    \advance\baselineskip\fontdimen12 \scriptfont\tw@
    \lineskip\thr@@\fontdimen8 \scriptfont\thr@@
    \lineskiplimit\lineskip
    \ialign{\hfil$\m@th\scriptstyle##$&$\m@th\scriptstyle{}##$\hfil\crcr
      #1\crcr
    }%
  }%
}
\makeatother

\theoremstyle{plain}
\newtheorem{theorem}{Theorem}[section]
\newtheorem{proposition}[theorem]{Proposition}
\newtheorem{lemma}[theorem]{Lemma}
\newtheorem{corollary}[theorem]{Corollary}
\newtheorem{conjecture}[theorem]{Conjecture}

\theoremstyle{definition}

\usepackage{hyperref}
\usepackage{xcolor}
\hypersetup{
    colorlinks,
    linkcolor={red!50!black},
    citecolor={blue!50!black},
    urlcolor={blue!80!black}
}

\newcommand{\p}{{\scriptscriptstyle+}}
\newcommand{\pp}{{\scriptscriptstyle++}}
\newcommand{\tp}{{\scriptscriptstyle\mathsf{T}}}
\newcommand{\hm}{{\scriptscriptstyle\mathsf{H}}}

\let\O\undefined
\let\S\undefined
\DeclareMathOperator{\O}{O}
\DeclareMathOperator{\U}{U}
\DeclareMathOperator{\V}{V}
\DeclareMathOperator{\proj}{proj}
\DeclareMathOperator{\rank}{rank}
\DeclareMathOperator{\S}{S}
\DeclareMathOperator{\SO}{SO}

\DeclareMathOperator{\Skew}{\mathsf{\Lambda}}
\DeclareMathOperator{\Sym}{\mathsf{S}}
\DeclareMathOperator{\Herm}{\mathsf{H}}
\DeclareMathOperator{\tr}{tr}
\DeclareMathOperator{\Gr}{Gr}
\DeclareMathOperator{\diag}{diag}

\DeclareMathOperator{\spn}{span}
\DeclareMathOperator{\inv}{inv}
\DeclareMathOperator{\initial}{in}

\allowdisplaybreaks

\begin{document}
\title{Degree of the Grassmannian as an affine variety}
\author[L.-H.~Lim]{Lek-Heng~Lim}
\address{Computational and Applied Mathematics Initiative, Department of Statistics,
University of Chicago, Chicago, IL 60637-1514.}
\email{lekheng@uchicago.edu}
\author[K.~Ye]{Ke Ye}
\address{KLMM, Academy of Mathematics and Systems Science, Chinese Academy of Sciences, Beijing 100190, China}
\email{keyk@amss.ac.cn}

\begin{abstract}
The degree of the Grassmannian with respect to the Pl\"ucker embedding is well-known. However, the Pl\"ucker embedding, while ubiquitous in pure mathematics, is almost never used in applied mathematics. In applied mathematics, the Grassmannian is usually embedded as projection matrices $\Gr(k,\mathbb{R}^n) \cong \{P \in \mathbb{R}^{n \times n} : P^\tp = P = P^2,\; \tr(P) = k\}$ or as involution matrices $\Gr(k,\mathbb{R}^n) \cong\{X \in  \mathbb{R}^{n \times n} : X^\tp = X, \; X^2 = I,\; \tr(X)=2k - n\}$. We will determine an explicit expression for the degree of the Grassmannian with respect to these embeddings. In so doing, we resolved a conjecture of Devriendt, Friedman, Reinke, and Sturmfels about the degree of $\Gr(2, \mathbb{R}^n)$ and in fact generalized it to $\Gr(k, \mathbb{R}^n)$. We also proved a set-theoretic variant of another conjecture of theirs about the limit of $\Gr(k,\mathbb{R}^n)$ in the sense of Gr\"obner degeneration.
\end{abstract}

\maketitle

\section{Introduction}

The standard way to embed a Grassmannian in an ambient space is the celebrated Pl\"ucker embedding, $\pi : \Gr(k, \mathbb{R}^n) \to \mathbb{P}(\Skew^k( \mathbb{R}^n))$, $\spn(v_1, \dots, v_k) \mapsto [ v_1 \wedge \dots \wedge v_k ]$.  The Pl\"ucker embedding has many appealing features, e.g., its mean curvature vanishes and so its image is a minimal submanifold \cite{AG12,BN91}; in addition it is a \emph{minuscule embedding} \cite{LM02}. However, there are  several difficulties if one attempts to use the image of the Pl\"ucker embedding as a model for the Grassmannian in applied mathematics. One issue is that while it maps subspaces to antisymmetric $k$-tensors, it does so only up to scaling, i.e., the image of $\pi$ is a projective variety. This presents a problem as equivalence classes can be tricky to implement well in software, and is the whole reason why one needs a model in applied mathematics for the Grassmannian that realizes abstract subspaces as concrete objects with coordinates. Another issue is the exceedingly high dimension $\binom{n}{k}$ of the ambient space $\mathbb{P}(\Skew^k (\mathbb{R}^n))$, compared to its intrinsic dimension of $k(n-k)$. For instance, while one might get around the first issue by  further embedding  $\mathbb{P}(\Skew^k (\mathbb{R}^n))$ into $\mathbb{R}^m$ with $m =\binom{\binom{n}{k} + 1}{2}$ (note also that this is only possible over $\mathbb{R}$ but not over $\mathbb{C}$), the high dimension $m$ becomes a liability when one needs to perform computations.

As a result, in areas connected to applications such as coding theory \cite{CHRSS,Conway}, machine learning \cite{EG}, optimization \cite{SSZ,Zhao}, and statistics \cite{Chi}, the Grassmannian is typically modeled as a set of projection matrices:
\begin{align}
\Gr(k,\mathbb{R}^n) &\cong \{P \in \Sym^2(\mathbb{R}^n) : P^2 = P,\; \tr(P) = k\};			\label{eq:proj}
\intertext{or, more recently, as a set of involution matrices \cite{ZLK20}:}
\Gr(k,\mathbb{R}^n) &\cong \{X \in \Sym^2(\mathbb{R}^n) : X^2  =I_n,\; \tr(X)=2k - n\}		\label{eq:inv}
\end{align}
within the vector space of real symmetric matrices $\Sym^2(\mathbb{R}^n)$.
Even without taking into account constraints that further limit dimension, points on the Grassmannian are now realized as $n \times n$ symmetric matrices, a far lower dimensional ambient space compared to that in the Pl\"ucker embedding. It is worth noting that the model in \eqref{eq:proj} is not limited to applied areas but is also common in geometric measure theory \cite{Mattila} and differential geometry \cite{Nicolaescu}.

By ``the degree of Grassmannian'' in the title, we meant the degree of either \eqref{eq:proj} or \eqref{eq:inv} as defined by  Hilbert polynomials of their projective closures, which notably applies to arbitrary fields \cite{Hartshorne77}. It is easy to see that \eqref{eq:proj} and \eqref{eq:inv} have the same degree but \eqref{eq:inv} is defined by simpler equations (see Section~\ref{sec:cplx}). Henceforth we will adopt the model \eqref{eq:inv} but our results will apply to \eqref{eq:proj} as well.

To determine the value of the degree of \eqref{eq:inv}, it is easier to work over $\mathbb{C}$ and use the fact that the degree of \eqref{eq:inv} is equal to that of its \emph{complex locus}
\begin{equation}\label{eq:clocus}
\Gr_\mathbb{C}(k,\mathbb{R}^n) =  \{X \in \Sym^2(\mathbb{C}^n) : X^2  =I_n,\; \tr(X)=2k - n\}.
\end{equation}
This is also the approach used in \cite{KHB24}.

However the complex locus is not the only complex geometric object that may be associated with $\Gr(k,\mathbb{R}^n) $. Indeed the \emph{complex Grassmannian} (in the involution model)
\begin{equation}\label{eq:cgrass}
\Gr(k,\mathbb{C}^n) = \{X \in \Herm^2(\mathbb{C}^n) : X^2 =I_n,\; \tr(X)=2k - n\}
\end{equation}
is arguably a more natural object. While \eqref{eq:clocus} defines a complex affine variety in the \emph{complex} vector space of complex symmetric matrices $\Sym^2(\mathbb{C}^n)$, \eqref{eq:cgrass} defines a real affine variety in the \emph{real} vector space of Hermitian matrices $\Herm^2(\mathbb{C}^n)$.

The authors of \cite{KHB24} favor \eqref{eq:clocus} over \eqref{eq:cgrass} but did not provide a rationale. In Section~\ref{sec:cplx}, we will show that the former possesses a special property --- \eqref{eq:clocus} gives a \emph{minimal algebraic complexification}, which is unique among all complexifications if one exists.  On the other hand \eqref{eq:cgrass} gives a nonminimal complexification. 

As our title suggests, our main goal is to establish an explicit expression for the degree of the Grassmannian as an embedded variety, which we accomplish in Section~\ref{sec:deg}. We will prove a closed-form combinatorial formula for the degree of \eqref{eq:clocus} in Theorem~\ref{thm:degree} and highlight some of its consequences. Notably, Corollary~\ref{cor:degreek=2} resolves \cite[Conjecture~5.7]{KHB24} and Corollaries~\ref{cor:degreek=34} and \ref{cor:any degree} confirm the numerical values  in \cite[Proposition~5.5]{KHB24} (that the authors computed with \texttt{Macaulay2} and \texttt{HomotopyContinuation.jl}). In Section~\ref{sec:proj}, we will characterize the boundary points of the projective closure of \eqref{eq:clocus} in Theorem~\ref{thm:set} and thereby resolve a set-theoretic version of \cite[Conjecture~5.8]{KHB24}.

\subsection{Degree for practitioners}\label{sec:prac}

We add a few words for computational and applied mathematicians who may use the models \eqref{eq:proj} or \eqref{eq:inv} in their works but who may not be familiar with the notion of \emph{degree of an algebraic variety} \cite{Hartshorne77,GH94}. As the name implies, it is a notion that generalizes the degree of a single polynomial to more than one  polynomials $f_1,\dots,f_k \in \mathbb{C}[x_1,\dots,x_n]$, or, equivalently, to the variety cut out by these polynomials $V \coloneqq \{x \in \mathbb{C}^n : f_1(x) = \dots = f_k(x) = 0\}$.

Geometrically, the degree counts the number of intersection points of the variety with a generic linear space of complementary dimension. Roughly speaking, it  counts the number of solutions of a linear system $A x = b$ on a variety $V \subseteq \mathbb{C}^n$. If $A \in \mathbb{C}^{m \times n}$ has full rank, $b \in \mathbb{C}^m$, and $V \subseteq \mathbb{C}^n$ is a variety of codimension $m$ and degree $d$, then by definition there will be at most $d$ solutions $x \in V$. The degree of a variety is also a measure of how complicated the variety is. For example, hypersurfaces of  degree one or two are easily understood whereas degree three or higher hypersurfaces are still mysterious \cite{Huybrechts23}. 

The degree is an invariant of a variety but it depends on the embedding. So the notion is especially pertinent to practitioners as algebraic varieties in applications are usually explicitly embedded in some ambient spaces like $\mathbb{C}^n$, $\mathbb{C}^{m \times n}$, $\Sym^2(\mathbb{C}^n)$, $\Skew^k (\mathbb{C}^n)$, etc. Our results on the degree of Grassmannian, while primarily of theoretical interest, have some practical implications. For example, if we optimize a generic degree-$p$ polynomial function on $\Gr(k,\mathbb{R}^n)$ in the models \eqref{eq:proj} or \eqref{eq:inv}, then the number of critical points is bounded above by  $p^{k(n-k)}d$ where $d$ is the degree of $\Gr_\mathbb{C}(k,\mathbb{R}^n)$ in \eqref{eq:clocus}.
For another example, the aforementioned problem of solving a system of linear equations on a variety $V \subseteq \mathbb{C}^n$ arises in \emph{unlabeled sensing} \cite{sense} where $V$  is a set of $n!$ points; this could conceivably be extended to $V = \Gr_\mathbb{C}(k,\mathbb{R}^n) \subseteq \Sym^2(\mathbb{C}) $.

The last section will involve the notion of a \emph{projective closure} of an affine variety. This is a standard  procedure to turn an affine variety in $\mathbb{C}^n$ into a projective variety in $\mathbb{P}(\mathbb{C}^{n+1})$ by adding ``points at infinity.'' Taking projective closure preserves degree, a fact that is often used in the calculation of degrees.

\section{Notations and background}

We write $\mathbb{Z}_\p $ for the set of nonnegative integers and $\mathbb{Z}_\pp $ for the set of positive integers throughout. For easy reference, we recall three results from linear algebra, representation theory, and combinatorics that we will need later.

\subsection{Linear algebra}

While a real symmetric matrix is orthogonally diagonalizable, a complex symmetric matrix is only  similar to a block diagonal matrix under conjugation by complex orthogonal matrices. We will use the following result from \cite[p.~13]{Gantmacher59}.
\begin{lemma}[Canonical form for complex symmetric matrices]\label{lem:canonical form}
Let $A\in \Sym^2(\mathbb{C}^n)$. There exists $Q\in \O_n(\mathbb{C})$ so that
\[
A =Q \diag(
\lambda_1 I_{q_1} + S_1, \dots, \lambda_k I_{q_m} + S_m) Q^\tp,
\]
a block diagonal matrix with diagonal blocks of the form
\[
S_j = \frac{1}{2} (I_{q_j} - i J_{q_j}) N_{q_j} (I_{q_j} + i J_{q_j}) \in \Sym^2(\mathbb{C}^{q_j}),\quad j=1,\dots,m,
\]
where $I_q$ is the $q \times q$ identity matrix and
\[
J_q \coloneqq \begin{bmatrix}
0 & 0 & \dots & 0 & 1 \\
0 & 0 & \dots & 1 & 0 \\
\vdots & \vdots & \ddots & \vdots & \vdots \\
0 & 1 & \dots & 0 & 0 \\
1 & 0 & \dots & 0 & 0 
\end{bmatrix} \in \Sym^2(\mathbb{C}^q),\quad  
N_q \coloneqq  \begin{bmatrix}
0 & 1 & \dots & 0 & 0 \\
0 & 0 & \dots & 0 & 0 \\
\vdots & \vdots & \ddots & \vdots & \vdots \\
0 & 0 & \dots & 0 & 1 \\
0 & 0 & \dots & 0 & 0 
\end{bmatrix} \in \mathbb{C}^{q \times q}
\]
are the $q \times q$ exchange matrix and nilpotent matrix respectively.
Here $\lambda_1,\dots, \lambda_m$, not necessarily distinct, are eigenvalues of $A$, and $q_1 + \dots + q_m = n$.
\end{lemma}

\subsection{Representation theory}

Irreducible $\SO_n(\mathbb{C})$-modules are indexed by non-increasing sequences $\lambda = (\lambda_1,\dots, \lambda_m)$ of $m = \lfloor n/2 \rfloor$ integers such that $\lambda_m \ge 0$ if $n = 2m + 1$ and $\lambda_{m-1} \ge |\lambda_m|$ if $n = 2m$. Let $\mathbb{V}_\lambda$ be the irreducible $\SO_n(\mathbb{C})$-module indexed by $\lambda$. Then its dimension \cite[Proposition~3.1.19]{GW09}  is given by
\begin{equation}\label{eq:dimension}
\dim \mathbb{V}_{\lambda} = \begin{cases}
\prod\limits_{1 \le i < j \le m} \dfrac{\lambda_i - \lambda_j - i +  j}{j - i} \prod\limits_{1 \le i \le j \le m} \dfrac{\lambda_i + \lambda_j + n - i -  j}{n - i- j} &\text{if } n = 2m+1, \\[4ex]
\prod\limits_{1 \le i < j \le m} \dfrac{\lambda_i - \lambda_j - i +  j}{j - i} \dfrac{\lambda_i + \lambda_j + n - i -  j}{n - i- j} &\text{if } n = 2m.
\end{cases}
\end{equation}

Let $m = \lfloor n/2 \rfloor$ and $e_1,\dots, e_m \in \mathbb{R}^m$  be the standard basis vectors. Then the \emph{fundamental weights} of $\SO_n(\mathbb{C})$ are $\omega_1,\dots, \omega_m \in \mathbb{R}^m$ defined by 
\[
\omega_i = \begin{cases}
e_1 + \dots + e_i &\text{if } n = 2m+1\text{ and }1\le i\le m-1, \\[1ex]
\frac{1}{2} (e_1 + \dots + e_m) &\text{if } n = 2m+1\text{ and }i = m, \\[1ex]
e_1 + \dots + e_i &\text{if } n = 2m\text{ and }1\le i\le m-2, \\[1ex]
\frac{1}{2} (e_1 + \dots + e_{m-1} - e_m) &\text{if } n = 2m\text{ and }i = m-1, \\[1ex]
\frac{1}{2} (e_1 + \dots + e_{m-1} + e_m) &\text{if } n = 2m\text{ and }i = m.
\end{cases}
\]

\subsection{Combinatorics}

The \emph{dominance partial ordering} $\succeq$ on the set of partitions is defined by 
\[
\lambda \succeq \mu\quad \iff \quad |\lambda| = |\mu|\text{ and }\sum_{j=1}^i \lambda_j \ge \sum_{j=1}^i \mu_j \text{ for each } i=1,\dots, m.
\]
The following expression may be found in \cite[Lemma~3.2]{Binegar08}.
\begin{lemma}\label{lem:integral}
Let $m,p,d \in \mathbb{Z}_\pp $ and $\delta \coloneqq (m-1, \dots, 1, 0)$. Then 
\[
\int_{|x|\le 1, \; x_1 \ge \dots \ge x_m \ge 0} \prod_{1 \le i \le m} x_i^p \prod_{1 \le i < j \le m} (x_i^2 - x_j^2)^d \, dx = \frac{\sum_{\lambda \succeq d\delta} A_{\lambda,p,d}  B_{\lambda,d} C_{\lambda,d}}{\Gamma(m(p + 1 + d(m-1))+1)},
\]
where
\[
A_{\lambda,p,d} = \prod_{i=1}^m \Gamma(\lambda_i + p + 1 + d(m-i)/2),\quad B_{\lambda,d} = \prod_{1\le i < j \le m} \frac{\Gamma(\lambda_i - \lambda_j + d(j-i+1)/2)}{\Gamma(\lambda_i - \lambda_j + d(j-i)/2)},
\] 
and, for $\lambda \succeq d\delta \coloneqq \bigl(d(m-1),d(m-2),\dots, d,0\bigr)$, $C_{\lambda,d}$ is the coefficient of the Jack symmetric functions $J_{\lambda}^{( \frac{2}{d})}(x)$ in the expansion  
\[
\prod_{1 \le i < j \le m} (x_i + x_j)^d = \sum_{\lambda \succeq d \delta} C_{\lambda,d} J_{\lambda}^{( \frac{2}{d} )}(x).
\]
\end{lemma} 
It is perhaps worth noting that Jack symmetric functions, first introduced in \cite{Jack70} as a generalization of Schur functions, have profound connections to combinatorics and representation theory \cite{Foulkes72,Mac,Stanley89,Cherednik95,EK96,Kirillov97}, and are closely related to various forms of the Selberg integral \cite{Kaneko93,Kadell97,Binegar08}. 

\section{Complex locus of the real Grassmannian}\label{sec:cplx}

A reason we favor our \emph{involution model} \eqref{eq:inv} over the \emph{projection model} \eqref{eq:proj} is that we find $X^2 = I$ more convenient to handle than $P^2 = P$. Since the projection model \eqref{eq:proj} is easily seen to be a scaled and translated copy of the involution model \eqref{eq:inv}, they have the same degree.  Henceforth we will assume the form in \eqref{eq:inv}. We begin by deriving its complex locus, showing that it is indeed given by \eqref{eq:clocus} as expected.  To that end, we will need to determine the ideal of $\Gr(k,\mathbb{R}^n)$.  Proposition~\ref{prop:ideal} below is the involution model analogue of \cite[Theorem~5.1]{KHB24}  for the projection model; and we give an elementary proof with classical invariant theory, avoiding the scheme theory used in \cite{KHB24}.
\begin{proposition}\label{prop:ideal}
For any $k,n \in \mathbb{Z}_\pp $ with $k\le n$, let $\mathscr{I}_{k,n}$ be the ideal generated by $2k - n - \tr(X)$ and $I_n - X^2$. Then $\mathscr{I}_{k,n} = \mathscr{I}(\Gr(k,\mathbb{R}^n))$.
\end{proposition}
\begin{proof}
Clearly $\mathscr{I}_{k,n} \subseteq \mathscr{I}(\Gr(k,\mathbb{R}^n)) \subseteq \mathbb{R}[\Sym^2(\mathbb{R}^n)]$. Let $\V_{n,k}(\mathbb{R})$ be the Stiefel variety of $k$ orthonormal frames in $\mathbb{R}^n$. The coordinate ring of $\V_{n,k}(\mathbb{R})$ is $\mathbb{R}[\V_{n,k}(\mathbb{R})] = \mathbb{R}[Y]/\langle  I_k - Y^\tp Y \rangle$, where $Y = (y_{i \ell})$ is the $n\times k$ matrix with indeterminate entries $y_{i \ell}$, $i = 1,\dots, n$, $\ell = 1,\dots, k$.  
Since $\Gr(k,\mathbb{R}^n) \cong \V_{n,k}(\mathbb{R})/\O_k(\mathbb{R})$, the coordinate ring may be determined as a ring of $\O_k(\mathbb{R})$-invariants,
\[
\mathbb{R}[\Gr(k,\mathbb{R}^n)] \simeq \mathbb{R}[\V_{n,k}(\mathbb{R})]^{\O_k(\mathbb{R})} = (\mathbb{R}[Y]/\langle  I_k - Y^\tp Y \rangle )^{\O_k(\mathbb{R})},
\]
where the action of $\O_k(\mathbb{R})$ is given by
\[
\O_k(\mathbb{R}) \times ( \mathbb{R}[Y]/\langle  I_k - Y^\tp Y \rangle  ) \to \mathbb{R}[Y]/\langle I_k - Y^\tp Y \rangle, \quad (Q,f(Y))\mapsto f(YQ^\tp).
\]
We consider the map
\[
\varphi: \mathbb{R}[X]/ \langle \tr(X) - k, X - X^2 \rangle \to ( \mathbb{R}[Y]/\langle  I_k -Y^\tp Y \rangle )^{\O_k(\mathbb{R})}, \quad \varphi (X) = Y Y^\tp.
\]
It is straightforward to verify that $\varphi$ is well-defined. The domain of $\varphi$ may also be written as
\[
\mathbb{R}[x_{ij} :i,j = 1,\dots,n]\!\!\biggm/\!\!\!\biggl\langle \sum_{i=1}^n x_{ii} -k,  x_{i j} - \sum_{\ell =1}^n x_{i \ell} x_{\ell j} : i,j = 1,\dots,n \biggr\rangle
\]
and in which case $\varphi$ takes $x_{ij}$ to  $\sum_{\ell=1}^k y_{i\ell}y_{j \ell}$,  $i,j = 1,\dots,n$. We claim that
$(\mathbb{R}[Y]/\langle  I_k -Y^\tp Y \rangle )^{\O_k(\mathbb{R})}$ is generated by
\[
 z_{ij}\coloneqq \sum_{\ell=1}^k y_{i\ell}y_{j \ell}, \quad i,j = 1,\dots,n,
\]
with relations
\[
z_{ij} =\sum_{\ell =1}^n z_{i \ell} z_{\ell j},\quad \sum_{i=1}^n z_{ii} = k,\quad i,j = 1,\dots,n,
\]
from which it follows that $\varphi$ is an isomorphism. The construction of these generators and relations is routine and we will just present a sketch: Let $G_1$ be the finite subgroup of $\O_k(\mathbb{R})$ consisting of all signed permutation matrices and $G_2 \coloneqq \bigl\{\begin{bsmallmatrix} Q & 0 \\ 0 & I_{k-2} \end{bsmallmatrix} \in \O_k(\mathbb{R}) : Q \in \O_2(\mathbb{R}) \bigr\}$. Note that $G_1 \cup G_2$ generates $\O_k(\mathbb{R})$. We first determine $( \mathbb{R}[Y]/\langle  I_k -Y^\tp Y \rangle )^{G_1}$ using the Reynold operator \cite{Pro}. By further imposing $G_2$-invariance on the elements of $( \mathbb{R}[Y]/\langle  I_k -Y^\tp Y \rangle )^{G_1}$, we see that
\[
(\mathbb{R}[Y]/\langle  I_k - Y^\tp Y \rangle )^{\O_k(\mathbb{R})} = ( \mathbb{R}[Y]/\langle  I_k -Y^\tp Y \rangle )^{G_1} \cap ( \mathbb{R}[Y]/\langle  I_k -Y^\tp Y \rangle )^{G_2} 
\]
is generated by $z_{ij}$, $i,j = 1,\dots,n$. To obtain the relations among these generators, let $f$ be a polynomial in $z_{ij}$, $i,j = 1,\dots,n$, and write $Z = [z_{ij}]$. Observe that if $f(Z) = 0$, then $f(Y^\tp Y) \in \langle I_k - Y^\tp Y \rangle$ and is invariant under both $G_1$ and $G_2$.  Since 
\begin{align*}
\sum_{i = 1}^n z_{ii} - k =\sum_{p=1}^k \biggl( \sum_{i=1}^n y_{ip}y_{ip} - 1 \biggr)  &\in \langle I_k - Y^\tp Y \rangle,\\
z_{ij} - \sum_{\ell=1}^n z_{i \ell} z_{\ell j} = \sum_{p =1}^k \sum_{q = 1}^k  y_{ip} \biggl( \delta_{pq} -  \sum_{\ell = 1}^n y_{\ell p}  y_{\ell q} \biggr) y_{j q} &\in \langle I_k - Y^\tp Y \rangle,
\end{align*}
we have
\[
f(Z) \in \biggl\langle z_{ij} - \sum_{\ell =1}^n z_{i \ell} z_{\ell j},  \sum_{i = 1}^n z_{ii} - k \biggr\rangle = \langle Z - Z^2,  \tr(Z) - k \rangle.
\]
Hence $\varphi$ gives an isomorphism
\[
( \mathbb{R}[Y]/\langle   I_k - Y^\tp Y \rangle )^{\O_k(\mathbb{R})} \simeq \mathbb{R}[X]/ \langle X - X^2, \tr(X) - k \rangle.
\]
This implies that the ideal of $\Gr(k,\mathbb{R}^n)$ via the embedding 
\[
j_{\proj}:\Gr(k,\mathbb{R}^n) \hookrightarrow \Sym^2(\mathbb{R}^n), \quad \mathbb{V} \mapsto V V^\tp
\]
is generated by $k - \tr(X)$ and $X - X^2$. Here $V$ is any representative of $\mathbb{V}$ in $\V_{n,k}(\mathbb{R})$. Since the involution model is obtained by composing $j_{\proj}$ with a translation, i.e., 
\[
j_{\inv}: \Gr(k,\mathbb{R}^n) \hookrightarrow \Sym^2(\mathbb{R}^n), \quad \mathbb{V} \mapsto 2 VV^\tp - I_n,
\] 
we conclude that the ideal of $j_{\inv} (\Gr(k,\mathbb{R}^n))$ in $\Sym^2(\mathbb{R}^n)$ is generated by $2k - n - \tr(X)$ and $I_n - X^2 $.
\end{proof}

It follows from Proposition~\ref{prop:ideal} that the complex locus of $\Gr(k,\mathbb{R}^n)$ is given by \eqref{eq:clocus}, i.e., replacing $\mathbb{R}$ by $\mathbb{C}$ in \eqref{eq:inv}. By Lemma~\ref{lem:canonical form}, we may write $X\in \Sym^2(\mathbb{C}^n)$ as  
\[
X = Q \diag(\lambda_1 I_{q_1} + S_1,\dots, \lambda_m I_{q_m} + S_m) Q^\tp
\]
for some $Q\in \O_n(\mathbb{C})$ and symmetric matrices $S_1,\dots,S_m$ as defined therein. Thus $X^2 = I_n$ if and only if $m = n$, $q_j = 1$, and $\lambda_j = \pm 1$,  $S_j = 0$, $ j=1,\dots, n$. This observation leads to the following description of $\Gr_\mathbb{C}(k,\mathbb{R}^n)$.
\begin{lemma}\label{lem:symmetric space}
Let $\O_n(\mathbb{C})$ act on $\Sym^2(\mathbb{C}^n)$ by conjugation. Then $\Gr_\mathbb{C}(k,\mathbb{R}^n)$ is the $\O_n(\mathbb{C})$-orbit of $\diag(I_k,-I_{n-k})$ and we have isomorphisms 
\begin{equation}\label{lem:symmetric space:eq1}
\Gr_\mathbb{C}(k,\mathbb{R}^n) \cong \O_n(\mathbb{C})/(\O_k(\mathbb{C}) \times \O_{n-k}(\mathbb{C}) ) \cong \SO_n(\mathbb{C})/\S(\O_k(\mathbb{C}) \times \O_{n-k}(\mathbb{C})),
\end{equation}
where
\[
\S\bigl(\O_k(\mathbb{C}) \times \O_{n-k}(\mathbb{C}) \bigr) \coloneqq \{ (X,Y) \in \O_k(\mathbb{C}) \times \O_{n-k}(\mathbb{C}) : \det(X Y) = 1\}.
\]
It follows that the coordinate ring
\begin{equation}\label{lem:symmetric space:eq2}
\mathbb{C}[\Gr_\mathbb{C}(k,\mathbb{R}^n)] \simeq \mathbb{C}[\SO_n(\mathbb{C})/\S(\O_k(\mathbb{C}) \times \O_{n-k}(\mathbb{C}))].
\end{equation}
\end{lemma}

We will next show that the complex locus in \eqref{eq:clocus} has a rather unique property.
Recall that a complexification \cite{Kulkarni75} of a real manifold $M$ is a complex manifold $M_\mathbb{C}$ satisfying $M\subseteq M_\mathbb{C}$ and $M = \{x\in M_\mathbb{C}: \tau (x) = x\}$ for some \emph{conjugation}  $\tau$, i.e., an anti-holomorphic involution such that for every fixed point $x\in M_\mathbb{C}$ of $\tau$, there is a holomorphic coordinate system $(z_1,\dots, z_n)$ around $x$ with $\tau(z_1,\dots, z_n) = (\overline{z}_1,\dots, \overline{z}_n)$. A complexification $M_\mathbb{C}$ is \emph{minimal} if the inclusion $M \subseteq M_\mathbb{C}$ is a homotopy equivalence. It is well-known that any real manifold $M$ admits a minimal analytic complexification \cite{Kulkarni78,WB59} and any compact real manifold $M$ can be realized as the set of real points of some algebraic variety \cite{Nash52}. However, the combination of these two statements is false: It is not true that any compact real manifold admits a minimal algebraic complexification.

Although the complex locus of a compact real variety is obviously an algebraic complexification, it is not necessarily minimal. For example, $M^\varepsilon = V((x^2 + 2y^2 - 1)(2x^2 + y^2 -1) + \varepsilon)$ is a disjoint union of four ovals in $\mathbb{R}^2$ for small $\varepsilon >0$  \cite[Chapter~48]{Popescu16}. Its complex locus $M^\varepsilon_\mathbb{C}$ is a Riemann surface of genus three with four points removed  \cite[Section~6]{VO02}. Thus $M^\varepsilon_\mathbb{C}$ is homotopic to the one point union of nine circles, from which we may conclude that the inclusion $M^\varepsilon \hookrightarrow M^\varepsilon_\mathbb{C}$ is not a homotopy equivalence. This example indicates that the homogeneous space structure of $\Gr_\mathbb{C}(k,\mathbb{R}^n)$ in Lemma~\ref{lem:symmetric space} is essential below.
\begin{proposition}[Minimal algebraic complexification]\label{prop:minimal algebraic complexification}
The complex locus $\Gr_\mathbb{C}(k,\mathbb{R}^n)$ in \eqref{eq:inv} is a minimal affine algebraic complexification of $\Gr(k,\mathbb{R}^n)$. The complex Grassmannian $\Gr(k,\mathbb{C}^n)$ in \eqref{eq:cgrass} is a non-minimal complexification of $\Gr(k,\mathbb{R}^n)$.
\end{proposition}
\begin{proof}
Recall that as Lie groups, $\SO_n(\mathbb{C})$ is the complexification of $\SO_n(\mathbb{R})$. By the isomorphism \eqref{lem:symmetric space:eq1} and the fact that $\S(\O_k(\mathbb{C}) \times \O_{n-k}(\mathbb{C})) \cap \SO_n(\mathbb{R}) = \S(\O_k(\mathbb{R}) \times \O_{n-k}(\mathbb{R}))$, the first statement is a direct consequence of the proof of \cite[Theorem~5.1]{Kulkarni78}. On the other hand, the complex Grassmannian $\Gr(k,\mathbb{C}^n) \cong \U(n)/(\U(k) \times \U(n-k))$,  so  $\pi_1 \bigl( \Gr(k,\mathbb{C}^n) \bigr) = 0$. Since $\pi_1 (\Gr(k,\mathbb{R}^n)) = \mathbb{Z}_2$, the natural inclusion $\Gr(k,\mathbb{R}^n)  \hookrightarrow \Gr(k,\mathbb{C}^n)$ cannot be a homotopy equivalence.
\end{proof}

Note that the conjectures in \cite{KHB24} concern real Grassmannians. This is ultimately our main reason for favoring the complex locus $\Gr_\mathbb{C}(k, \mathbb{R}^n)$ over the complex Grassmannian $\Gr(k, \mathbb{C}^n)$, as the degree of the real Grassmannian $\Gr(k, \mathbb{R}^n)$ in $\Sym^2(\mathbb{R}^n)$ equals that of its complex locus $\Gr_\mathbb{C}(k, \mathbb{R}^n)$ in  $\Sym^2(\mathbb{C}^n)$. To elaborate, we will need a more formal definition of degree than that given in Section~\ref{sec:prac}: For a $p$-dimensional affine variety $V \subseteq \mathbb{R}^m$,  its degree,  denoted as $\deg V$, is defined as the number of intersection points of $V_\mathbb{C} \subseteq \mathbb{C}^m$,  a complex affine variety of complex dimension $p$, with a generic $(m - p)$-dimensional subspace in $\mathbb{C}^m$.  See,  for example,  \cite[p.~13]{KHB24} and \cite[p.~2]{BG21}. It turns out that $\deg V$ can be determined by the leading coefficient of the Hilbert polynomial of the projective closure of $V_\mathbb{C}$ \cite[p.~52]{Hartshorne77}, and  consequently,
\begin{equation}\label{eq:limit}
\deg V = p! \lim_{d \to \infty} \frac{\dim \mathbb{C}[V]_d}{d^p}.
\end{equation}

We do not know if the results in our article extend to complex Grassmannians. We suspect not, since the geometry of $\Gr(k, \mathbb{C}^n)$ and $\Gr_\mathbb{C}(k, \mathbb{R}^n)$ are vastly different. Indeed, there are more fundamental differences than that revealed by Proposition~\ref{prop:minimal algebraic complexification}. Firstly, $\Gr(k, \mathbb{C}^n)$ is compact whereas $\Gr_\mathbb{C}(k, \mathbb{R}^n)$ is not.  Secondly, $\Gr(k, \mathbb{C}^n)$ is a \emph{real} affine variety in the \emph{real} linear space $\Herm^2(\mathbb{C}^n)$ but $\Gr_\mathbb{C}(k, \mathbb{R}^n)$ is a complex affine variety in the complex linear space $\Sym^2(\mathbb{C}^n)$. Although $\Gr(k,\mathbb{C}^n)$, as the image of the Pl\"ucker embedding, is also a complex \emph{projective} variety, here we are mainly interested in the degree of the Grassmannian as an affine variety, as our title indicates.

Another useful consequence of Lemma~\ref{lem:symmetric space} is that it allows one to completely determine the decomposition of the coordinate ring $\mathbb{C}[\Gr_\mathbb{C}(k,\mathbb{R}^n)]$, which is an $\SO_n(\mathbb{C})$-module, into a direct sum of irreducible $\SO_n(\mathbb{C})$-submodules \cite[Corollary~12.3.15]{GW09}. 
\begin{proposition}\label{prop:symm}
Let $k, n \in \mathbb{Z}_\pp$ with $k \le n/2$. Then
\begin{equation}
\mathbb{C}[\Gr_\mathbb{C}(k,\mathbb{R}^n)] \simeq \bigoplus_{\lambda\in \Lambda_{k,n}} \mathbb{V}_\lambda
\end{equation}
where  $\Lambda_{k,n}$ is generated by $\omega_1,\dots, \omega_{\lfloor n/2 \rfloor}$, the fundamental weights of $\SO_n(\mathbb{C})$, as follows:
\[
\Lambda_{k,n} = \begin{cases}
\spn_{\mathbb{Z}_\p }\{2\omega_1,\dots, 2\omega_k \} & \text{if } n = 2m + 1\text{ and }k \le m-1,  \\ 
\spn_{\mathbb{Z}_\p }\{2\omega_1,\dots, 2\omega_{m-1},4\omega_m \} & \text{if }  n = 2m + 1\text{ and }k=m,  \\ 
\spn_{\mathbb{Z}_\p }\{2\omega_1,\dots, 2\omega_k \} & \text{if } n = 2m \text{ and }k \le m-2 \text{ or } k = m, \\ 
\spn_{\mathbb{Z}_\p }\{2\omega_1,\dots, 2\omega_{m-2},2\omega_{m-1} + 2\omega_m \} & \text{if } n = 2m \text{ and }k = m-1.
\end{cases}
\]
\end{proposition}

\section{Degree of $\Gr(k,\mathbb{R}^n)$}\label{sec:deg}

The involution model of $\Gr(k, \mathbb{R}^n)$ is linearly isomorphic to its projection model: $X = 2P - I$ and $P = (I + X)/2$ takes one back and forth between  \eqref{eq:proj} and \eqref{eq:inv}. As a result,  the degree  of $\Gr(k, \mathbb{R}^n)$ in the involution model is identical to that in the projection model and we have in effect resolved \cite[Conjecture~5.7]{KHB24}, reproduced below for easy reference and  formally stated as Corollary~\ref{cor:degreek=2} to our main result Theorem~\ref{thm:degree}.
\begin{conjecture}[Devriendt, Friedman, Reinke, and Sturmfels]
The degree of $\Gr(2, \mathbb{R}^n)$ in the projection model is  $2 \binom{2n-4}{n-2}$.
\end{conjecture}

As we noted earlier, the involution model of $\Gr(k,\mathbb{R}^n)$ as defined in \eqref{eq:inv} has degree equal to that of its complex locus  $\Gr_\mathbb{C}(k,\mathbb{R}^n)$ as defined in \eqref{eq:clocus}. Given that $\Gr_\mathbb{C}(k,\mathbb{R}^n)$ is a subvariety of $\Sym^2(\mathbb{C}^n)$, its coordinate ring $\mathbb{C}[\Gr_\mathbb{C}(k,\mathbb{R}^n)]$ is a quotient ring of the polynomial ring $\mathbb{C}[\Sym^2(\mathbb{C}^n)]$. For any $d \in \mathbb{Z}_\p $, we will write $\mathbb{C}[\Gr_\mathbb{C}(k,\mathbb{R}^n)]_d$ for the subspace of $\mathbb{C}[\Gr_\mathbb{C}(k,\mathbb{R}^n)]$ comprising functions that are restrictions of polynomials of degree at most $d$ in $\mathbb{C}[\Sym^2(\mathbb{C}^n)]$.  As a consequence of \eqref{eq:limit} and  Proposition~\ref{prop:symm},  we have the following: 
\begin{corollary}[Degree of Grassmannian as a limit]\label{cor:degree}
Let $k,n,\Lambda_{n,k}$ be as in Proposition~\ref{prop:symm}. Then for any  $d \in \mathbb{Z}_\p$,
\begin{equation}\label{cor:degree:eq1}
\mathbb{C}[\Gr_\mathbb{C}(k,\mathbb{R}^n)]_d \simeq \bigoplus_{\subalign{\lambda &\in \Lambda_{k,n} \\ |\lambda| &\le 2d} } \mathbb{V}_\lambda.
\end{equation}
Here if $n = 2k$, then $|\lambda| \coloneqq \lambda_1 + \dots + \lambda_{k-1} +|\lambda_k|$. The degree of $\Gr_\mathbb{C}(k,\mathbb{R}^n)$ in $\Sym^2(\mathbb{C}^n)$ is therefore given by
\begin{equation}\label{cor:degree:eq2}
d_{k,n} = p! \lim\limits_{d\to \infty} \frac{1}{d^p } \sum\limits_{\subalign{\lambda &\in \Lambda_{k,n} \\ |\lambda|  &\le 2d} } \dim \mathbb{V}_{\lambda}
\end{equation}
where $p \coloneqq k(n-k)$.
\end{corollary}

Let $k, n \in \mathbb{Z}_\pp $ with $k \le n/2$. We introduce the shorthand
\begin{equation}\label{eq:alphakn}
\alpha_{k,n} \coloneqq \begin{cases}
\dfrac{2^{k (n-k-1)}}{\prod_{\subalign{1 &\le i \le k\\ i &< j \le \frac{n}{2}}}  (j-i)(n-j-i)} &\text{if  $n$ is even and $k \le n/2 - 1$},\\[4ex]
\dfrac{2^{k (n-k)} }{\prod_{\subalign{1 &\le i \le k\\ i &< j \le \frac{n-1}{2}}} (j-i)(n - i - j) \prod_{1 \le i \le k} (n-2i)} &\text{if  $n$ is odd}, \\[5ex]
\dfrac{2^{k (k-1) +1} }{\prod_{1 \le i < j  \le k}  (j-i)(2k - j - i)} &\text{if  $n = 2k$}.
\end{cases}
\end{equation}

We now prove our main result.
\begin{theorem}[Degree of Grassmannian]\label{thm:degree}
For positive integers $k \le n$, the degree of $\Gr_\mathbb{C}(k,\mathbb{R}^n)$ in $\Sym^2(\mathbb{C}^n)$ is the same as that of $\Gr_\mathbb{C}(n-k,\mathbb{R}^n)$ in $\Sym^2(\mathbb{C}^n)$. For $k \le n/2$, this value is given by
\begin{equation}\label{eq:degree}
d_{k,n} = \alpha_{k,n} \sum_{\lambda \succeq \delta_k} A_{\lambda,k}  B_{\lambda,k} C_{\lambda,k},
\end{equation}
where $\delta_k = (k-1,\dots,1,0)$, 
\[
A_{\lambda,k} \coloneqq \prod_{i=1}^k \Gamma( n-2k + 1 + \lambda_i + (k-i)/2),\quad B_{\lambda,k} \coloneqq \prod_{1\le i < j \le k} \frac{\Gamma(\lambda_i - \lambda_j + (j-i+1)/2)}{\Gamma(\lambda_i - \lambda_j + (j-i)/2)},
\]
and, for $\lambda \succeq \delta_k$, $C_{\lambda,k}$ is the coefficient of Jack symmetric functions $J_{\lambda}^{(2)}(x)$ in the expansion 
\[
\prod_{1 \le i < j \le k} (x_i + x_j) = \sum_{\lambda \succeq \delta_k} C_{\lambda,k} J_{\lambda}^{(2)}(x).
\] 
\end{theorem}
\begin{proof} 
The equality between degrees of $\Gr_\mathbb{C}(k,\mathbb{R}^n)$ and $\Gr_\mathbb{C}(n-k,n)$, also found in \cite[Equation~8]{ZLK20} and \cite[Corollary~5.6]{KHB24}, follows from the isomorphism $\Sym^2(\mathbb{C}^n) \to \Sym^2(\mathbb{C}^n)$, $A \mapsto I_n - A$.

Recall that if $k < n/2$, a partition $\mu  = (\mu_1,\dots, \mu_m)$ lies in $\Lambda_{k,n}$ if and only if $\mu_{1},\dots, \mu_{k}\in 2\mathbb{Z}_\p $ and $\mu_{k
+1} = \dots = \mu_{m} = 0$; whereas for $n = 2k$, $\mu = (\mu_1,\dots, \mu_k)$ lies in $\Lambda_{k,2k}$ if and only if $\mu_1,\dots, \mu_k$ are of the same parity and $\mu_1 \ge \dots \ge \mu_{k-1} \ge |\mu_k|$. With this in mind, we consider three cases with respect to the values of $n$ and $k$.

\textsc{Case I:} $n = 2m$, $k \le m - 1$.\; By \eqref{eq:dimension}, we have
\begin{align*}
\dim \mathbb{V}_{\mu} &= \prod_{1 \le i < j \le k} \Bigl( \frac{\mu_i - \mu_j}{j-i} + 1 \Bigr) \Bigl( \frac{\mu_i + \mu_j}{n - j-i} + 1 \Bigr) \prod_{\subalign{1 &\le i \le k \\ k + 1 &\le j \le m}} \Bigl( \frac{\mu_i}{j-i} + 1 \Bigr) \Bigl( \frac{\mu_i }{n - j-i} + 1 \Bigr) \\
&= \prod_{1 \le i < j \le k}  \frac{\mu^2_i - \mu^2_j}{(j-i)(n - j-i)}    
\prod_{\subalign{1 &\le i \le k \\ k + 1 &\le j \le m}} \frac{\mu^2_i}{(j-i)(n - j-i)} +\text{ lower order terms}  \\
&=\frac{1}{D_{k,n}} \prod_{1 \le i < j \le k}  (\mu^2_i - \mu^2_j)
\biggl[ \prod_{1\le i \le k } \mu_i \biggr]^{n-2k} + \text{ lower order terms}
\end{align*}
where
\[
D_{k,n} \coloneqq 
{\prod\limits_{\subalign{1 &\le i \le k \\ i &< j \le m}}  (j-i)(n-j-i)}.
\]
By \eqref{cor:degree:eq2}, we have 
\begin{align*}
d_{k,n} &= p! \lim\limits_{d\to \infty} \frac{1}{d^p } \sum\limits_{\subalign{\lambda &\in \Lambda_{k,n} \\ |\lambda|  &\le 2d} } \dim \mathbb{V}_{\lambda} \\
&= \frac{2^p p!}{D_{k,n}}  \lim\limits_{d\to \infty} \frac{1}{(2d)^p } \sum\limits_{\subalign{\lambda &\in \Lambda_{k,n} \\ |\lambda| &\le 2d} } \prod_{1 \le i < j \le k}  (\lambda^2_i - \lambda^2_j)
\biggl[\prod_{1\le i \le k } \lambda_i \biggr]^{n - 2k} \\
&=\frac{2^p p!}{D_{k,n}}  \lim\limits_{d\to \infty} \frac{1}{(2d)^k } \sum\limits_{\subalign{\lambda &\in \Lambda_{k,n} \\ |\lambda| &\le 2d} } \prod_{1 \le i < j \le k}  \biggl[ \biggl(\frac{\lambda_i}{2d}\biggr)^2 - \biggl( \frac{\lambda_j}{2d} \biggr)^2 \biggr]
\biggl[ \prod_{1\le i \le k } \frac{\lambda_i}{2d} \biggr]^{n - 2k} \\
&=\frac{2^p p!}{D_{k,n}}  \lim\limits_{d\to \infty} \frac{1}{(2d)^k } \sum\limits_{\subalign{2d t &\in \Lambda_{k,n} \\ |t| &\le 1} } \prod_{1 \le i < j \le k} ( t_i^2 - t_j^2 )
\biggl[ \prod_{1\le i \le k } t_i \biggr]^{n - 2k} \\
&= \frac{2^{p-k} p!}{D_{k,n}} \int_{\subalign{|t| &\le 1 \\ 0 &\le t_k \le \dots \le t_1 \le 1}} 
\biggl[ \biggl( \prod_{1\le i \le k } t_i \biggr)^{n-2k} \prod_{1 \le i < j \le k}  ( t_i^2 - t_j^2 ) \biggr] dt.
\end{align*}

\textsc{Case II:} $n = 2m+1$,  $k \le m$.\; The dimension formula \eqref{eq:dimension} gives
\begin{align*}
\dim \mathbb{V}_{\mu} &= \prod\limits_{1 \le i < j \le m} \Bigl( \frac{\mu_i - \mu_j}{j - i} + 1\Bigr) \prod\limits_{1 \le i \le j \le m} \Bigl( \frac{\mu_i + \mu_j}{n - i- j} + 1 \Bigr) \\
&= \prod\limits_{1 \le i < j \le k} \Bigl( \frac{\mu_i - \mu_j}{j - i} + 1\Bigr) 
 \prod\limits_{1 \le i \le j \le k} \Bigl( \frac{\mu_i + \mu_j}{n - i- j} + 1 \Bigr) 
 \prod\limits_{1 \le i \le k < j \le m} \Bigl( \frac{\mu_i}{j - i} + 1\Bigr) \Bigl( \frac{\mu_i}{n - i- j} + 1 \Bigr)\\
 &= \prod\limits_{1 \le i < j \le k} \frac{\mu_i^2 - \mu_j^2}{(j - i)(n-i-j)} 
 \prod\limits_{1 \le i \le k} \frac{2 \mu_i}{n - 2i}    \prod\limits_{1 \le i \le k < j \le m}  \frac{\mu_i^2}{(j - i)(n-i-j)}  + \text{lower order terms}\\
 &= \frac{2^k}{E_{k,n}}  \prod\limits_{1 \le i < j \le k} (\mu_i^2 - \mu_j^2)  \prod\limits_{1 \le i \le k} \mu_i^{2(m-k) + 1} + \text{lower order terms},
\end{align*}
where
\[
E_{k,n} \coloneqq \prod\limits_{\subalign{1 &\le i \le k \\ i &< j \le m }} (j-i)(n - i - j) \prod\limits_{1 \le i \le k} (n-2i).
\]
By \eqref{cor:degree:eq2} and the same calculation as in \textsc{Case~I}, we obtain 
\[
d_{k,n} = \frac{2^{p} p!}{E_{k,n}} \int_{\subalign{|t| &\le 1 \\ 0 &\le t_k \le \dots \le t_1 \le 1}} \biggl[ 
\biggl( \prod_{i=1}^k t_i \biggr)^{n - 2k} \prod_{1 \le i < j \le k} (t_i^2 - t_j^2)
\biggr] dt.
\]

\textsc{Case III:} $n = 2k$.\; We recall that in this case, $|\mu| = \mu_1 + \dots + \mu_{k-1} + |\mu_k|$. Let
\[
F_{k,2k} \coloneqq  \prod_{1 \le i < j \le k }  (j-i)(2k-j-i).
\]
By \eqref{cor:degree:eq2}, we have
\begin{align*}
d_{k,2k} &= \frac{(k^2)!}{F_{k,2k}}  \lim_{d\to \infty} \frac{1}{d^{k^2} } \sum_{\subalign{\lambda &\in \Lambda_{k,n} \\ |\lambda| &\le 2d} } \prod_{1 \le i < j \le k}  (\lambda^2_i - \lambda^2_j) \\
&= \frac{2^{k^2 - k} (k^2)!}{F_{k,2k}} \int_{\subalign{ t_1 + \dots + t_{k-1} + |t_k| &\le 1 \\ |t_k| \le t_{k-1} \le \dots \le t_1 &\le 1 }} 
\biggl[ \prod_{1 \le i < j \le k} (t_i^2 - t_j^2) \biggr] dt  \\
 &= \frac{2^{k^2 - k + 1} (k^2)!}{F_{k,2k}} \int_{\subalign{|t| &\le 1\\ 0 &\le t_k \le \dots \le t_1 \le 0}} \biggl[
 \prod_{1 \le i < j \le k} (t_i^2 - t_j^2) \biggr] dt.
\end{align*}

Applying  Lemma~\ref{lem:integral} to the last integral in each of the three cases yields the required expression in \eqref{eq:degree}.
\end{proof}

By Theorem~\ref{thm:degree}, it is immediate that for $k = 1$, we get $d_{1,n} = 2^{n-1}$, which is also obtained in \cite[Corollary~5.6]{KHB24} via a geometric argument. For $k = 2$, we confirm the value conjectured in \cite[Conjecture~5.7]{KHB24} (and verified numerically for $n \le 10$ therein):
\begin{corollary}[Degree of $\Gr(2,\mathbb{R}^n)$]\label{cor:degreek=2}
For $n \ge 3$, we have
\[
d_{2,n} = 2 \binom{2n - 4}{n-2}.
\]
\end{corollary}
\begin{proof}
For $k =2$, we have $\delta_2 = (1,0)$ and $\lambda \succeq \delta_2$ if and only if $\lambda = (1,0)$. Moreover,
\[
A_{(1,0),2} = \Gamma(n -\tfrac32) \Gamma(n-3),\quad B_{(1,0),2} = \frac{\Gamma(2)}{\Gamma(\frac{3}{2})},
\] 
and  since $J^{(2)}_{(1)}(x) = x_1 + x_2 $, we get $C_{(1,0),2} = 1$. Hence we obtain from \eqref{eq:degree} that $d_{2,n} = 2^{n-1}   (2n - 5)!! /(n-2)!  = 2 \binom{2n - 4}{n-2}$.
\end{proof}
For $k = 3$ and $4$, we may also simplify the expression in \eqref{eq:degree} to obtain more explicit ones for $d_{3,n}$ and $d_{4,n}$. They confirm the values obtained numerically for $n \le 10$ in \cite[Proposition~5.5]{KHB24}.
\begin{corollary}[Degrees of $\Gr(3,\mathbb{R}^n)$ and $\Gr(4,\mathbb{R}^n)$]\label{cor:degreek=34}
For $n \ge 5$, we have
\[
d_{3,n} = \frac{(8n-25) (2n-9)!!}{(n-2)!} 2^{2n - 6}.
\]
For $n \ge 7$, we have
\[
d_{4,n} = \frac{(32n^2 - 288n + 634 ) (2n-13)!! (2n-9)!!}{(n-2)!(n-4)!} 2^{2n - 6}.
\]
\end{corollary}
\begin{proof}
For $k = 3$, we have $\delta_3 = (2,1)$ and $\lambda \succeq \delta_3$ if and only if $\lambda = (2,1)$ or $(1,1,1)$. Moreover, 
\begin{align*}
A_{(2,1),3}  &= \Gamma(n - 2) \Gamma(n- \tfrac72)  \Gamma(n-5), \quad B_{(2,1),3}  = \frac{\Gamma(2)}{\Gamma(\frac32)} \frac{\Gamma(\frac72)}{\Gamma(3)} \frac{\Gamma(2)}{\Gamma(\frac32)} = \frac{15}{4\sqrt{\pi}},\\
A_{(1,1,1),3} &=  \Gamma(n-3) \Gamma(n- \tfrac72)  \Gamma(n- 4),\quad B_{(1,1,1),3} = \frac{\Gamma(1)}{\Gamma(\frac12)} \frac{\Gamma( 1 )}{\Gamma(\frac12)} \frac{\Gamma( \frac32)}{\Gamma(1)} = \frac{1}{2\sqrt{\pi}},
\end{align*}
and since 
\[
\prod_{1\le i < j \le 3} (x_i + x_j) = J^{(2)}_{(2,1)} + \frac{1}{2} J^{(2)}_{(1,1,1)},
\] 
we get $C_{(2,1),3} = 1$, $C_{(1,1,1),3} = \frac12$. Hence we obtain the expression for $d_{3,n}$ from \eqref{eq:degree}. The expression for $d_{4,n}$ is similarly obtained.
\end{proof}

As these calculations reveal, if not for the fact that  the coefficients  $\{C_{\lambda,k} :\lambda \succeq \delta_k\}$ are implicitly defined, our expression for $d_{k,n}$ in \eqref{eq:degree} will be fully explicit, as $\alpha_{k,n}$, $A_{\lambda, k}$, $B_{\lambda,k}$ are all explicitly given. While in general there is no explicit formula for the coefficients $C_{\lambda,k}$ in
\[
\prod_{1 \le i < j \le k} (x_i + x_j) = \sum_{\lambda \succeq \delta_k} C_{\lambda,k} J_{\lambda}^{(2)}(x),
\]
they are trivial to compute algorithmically. As described in \cite[page~326]{Mac}, a Jack symmetric function $J_{\lambda}^{(2)}(x)$ can be expanded as a linear combination of monomial symmetric functions using the recursive Gram--Schmidt process, which in turn yields the values of $\{C_{\lambda,k} :\lambda \succeq \delta_k \}$. They may also be calculated efficiently in floating point arithmetic to high accuracy \cite{demmel}.

Each expression in Corollaries~\ref{cor:degreek=2} and \ref{cor:degreek=34} is a product of factorials, double factorials, $2$-powers, and a polynomial in $n$.  We will show that this holds true in general for  $d_{k,n}$. In fact, the first three quantities can be determined explicitly. As in Theorem~\ref{thm:degree}, we may assume $k \le n/2$ without loss of generality.
\begin{corollary}\label{cor:any degree}
Let $k,n \in \mathbb{Z}_\pp $ with $k \le n/2$. Then there exists a polynomial $P_k$ of degree at most $\binom{k}{2} - \sum_{i=1}^k \lfloor i/2 \rfloor$ such that
\begin{equation}\label{cor:any degree:eq1}
d_{k,n} =\alpha_{k,n} \frac{\prod_{j=0}^{ \lfloor \frac{k-1}{2} \rfloor } (n-2k + 2j)! \prod_{j=1}^{\lfloor k/2\rfloor} (2(n - 2k + 2j) - 1)!! }{ 2^{\lfloor k/2 \rfloor (n-2k + \lfloor k/2 \rfloor - 1)}}   P_k(n),
\end{equation}
where $\alpha_{k,n}$ is as in \eqref{eq:alphakn}. Moreover, for any fixed $k \in \mathbb{Z}_\pp $, the sequence $(d_{k,n})_{n = 2k}^\infty$ is completely determined by its first $\binom{k}{2} - \sum_{i=1}^k \lfloor i/2 \rfloor$ terms.
\end{corollary}
\begin{proof}
The existence of the polynomial $P_k$ and the expression \eqref{cor:any degree:eq1} are a direct consequence of Theorem~\ref{thm:degree}. First observe that $C_{\lambda,k} $ is independent of $n$.  By plugging $\Gamma(m) = (m-1)!$ and $\Gamma\bigl((2m+1)/2\bigr) = (2m-1)!!\sqrt{\pi}/2^{m}$ into \eqref{eq:degree} and regrouping terms,  we obtain \eqref{cor:any degree:eq1}. If the first $\binom{k}{2} - \sum_{i=1}^k \lfloor i/2 \rfloor$ terms in $(d_{k,n})_{n =2k}^\infty$ are known, then the subsequent values of $P_k(n)$ can be uniquely determined by polynomial interpolation, and thereby determining the corresponding values of $d_{k,n}$ via \eqref{cor:any degree:eq1}.  
\end{proof}
We recall from \cite[Theorem~5.13]{MS21} that the degree of $\Gr(k,\mathbb{R}^n)$ in the Pl\"{u}cker embedding is
\begin{equation}\label{eq:pluckerdegree}
\widehat{d}_{k,n} \coloneqq \frac{(k(n-k))!}{\prod_{j=1}^k j (j+1) \cdots (j + n - k -1)}.
\end{equation}
Corollary~\ref{cor:any degree} then allows us to compare $d_{k,n}$ with $\widehat{d}_{k,n}$ for any fixed $k$.
\begin{proposition}[Comparison with degree of Pl\"ucker embedding]\label{prop:degreecomparison}
Let $k\in \mathbb{Z}_\p $ be fixed. Then $d_{1,n}/\widehat{d}_{1,n} = 2^{1-n}$, $d_{2,n}/\widehat{d}_{2,n} = \frac12 (n-1)^{-1}$, and, for $k \ge 3$,
\begin{equation}\label{eq:stir}
d_{k,n}/\widehat{d}_{k,n}  = O\bigl(  (2/k )^{kn} n^{k^2} \bigr).
\end{equation}
\end{proposition}
\begin{proof}
The values for $k=1,2$ follow from Corollary~\ref{cor:degreek=2} and the discussion before it. For $k \ge 3$, it follows from \eqref{eq:alphakn}, \eqref{cor:any degree:eq1}, and \eqref{eq:pluckerdegree} that
\[
d_{k,n}/\widehat{d}_{k,n} =  O\biggl( \frac{n^{k^2} \prod_{j=1}^{k/2} (2n - 4k + 4j - 1)!}{(kn - k^2)!} \biggr).
\]
Applying Stirling's formula gives us \eqref{eq:stir}.
\end{proof}
Proposition~\ref{prop:degreecomparison} shows that for any fixed $k \ge 3$,  the degree of $\Gr(k,\mathbb{R}^n)$ in the involution model is exponentially smaller than its degree with respect to the Pl\"{u}cker embedding,  i.e., $d_{k,n}/\widehat{d}_{k,n}$ decreases to $0$ exponentially as $n\to \infty$. The practical implication is that the involution model for $\Gr(k, \mathbb{R}^n)$ is geometrically much simpler than the Pl\"ucker embedding, and low-degree objects are always preferred in computations. 

\section{Projective closure of the Grassmannian}\label{sec:proj}

Our main goal in this section is to prove a set-theoretic version of \cite[Conjecture~5.8]{KHB24}, reproduced below for easy reference.
\begin{conjecture}[Devriendt, Friedman, Reinke, and Sturmfels]\label{conj:limit}
In the sense of Gr\"{o}bner degeneration with respect to the monomial order given by total degree, the limit of $\Gr_\mathbb{C}(\lfloor n/2 \rfloor, \mathbb{R}^n)$ is $ \{ X \in \Sym^2(\mathbb{C}^n) : X^2 = 0\}$. Furthermore, the initial ideal is given by $\initial(\mathscr{I}_{\lfloor n/2 \rfloor,n}) = \langle X^2,  \tr(X) \rangle$.
\end{conjecture}
The notion of Gr\"{o}bner degeneration is discussed in \cite{Bayer82, Eisenbud95,  CV20}.  By definition,  $\initial(\mathscr{I}_{\lfloor n/2 \rfloor,n})$ is the limit of $\mathscr{I}_{\lfloor n/2 \rfloor,n}$ with respect to the Gr\"{o}bner degeneration.  So Conjecture~\ref{conj:limit} may be rephrased as 
\begin{equation}\label{eq:conj}
Z(\initial(\mathscr{I}_{\lfloor n/2 \rfloor,n})) = \{X\in \Sym^2(\mathbb{C}^n): X^2 = 0 \},\quad 
\initial(\mathscr{I}_{\lfloor n/2 \rfloor,n}) = \langle X^2,  \tr(X) \rangle
\end{equation}
where $Z(\mathscr{I})$ denotes the variety defined by the ideal $\mathscr{I}$. 

We will prove a set-theoretic variant of Conjecture~\ref{conj:limit}. Instead of the limit of the ideal $\mathscr{I}_{\lfloor n/2 \rfloor,n}$, we will give the limit points of $\Gr_\mathbb{C}(\lfloor n/2 \rfloor,\mathbb{R}^n)$ in $\mathbb{P}(\Sym^2(\mathbb{C}^n) \oplus \mathbb{C})$. Theorem~\ref{thm:set} shows, among other things, that the set of limit points $\partial \Gr_\mathbb{C}(\lfloor n/2 \rfloor,\mathbb{R}^n)$ is exactly the conjectured $Z(\initial(\mathscr{I}_{\lfloor n/2 \rfloor,n}))$ in \eqref{eq:conj}, i.e.,
\[
\partial \Gr_\mathbb{C}(\lfloor n/2 \rfloor,\mathbb{R}^n)
\xlongequal{\text{Thm.~\ref{thm:set}}}
\{X\in \Sym^2(\mathbb{C}^n): X^2 = 0 \} 
\xlongequal{\text{Conj.\ \ref{conj:limit}}}
Z(\initial(\mathscr{I}_{\lfloor n/2 \rfloor,n})).
\]
In fact, Theorem~\ref{thm:set} shows that the first equality holds with $\lfloor n/2 \rfloor$ replaced by any  $k \le \lfloor n/2 \rfloor$.

Conjecture~\ref{conj:limit} and Theorem~\ref{thm:set} are both about the limiting behavior of $\Gr(k,\mathbb{R}^n)$ with ``limit'' interpreted respectively in the sense of Gr\"{o}bner degeneration and in the sense of topology. In this regard \eqref{eq:n/2} in Theorem~\ref{thm:set} may be viewed as a set-theoretic version of \eqref{eq:conj} or Conjecture~\ref{conj:limit}.

More generally Theorem~\ref{thm:set} gives the polynomial equations defining $\overline{\Gr(k,\mathbb{R}^n)}$. Such set-theoretic descriptions of a variety are a common step towards the (usually more difficult) ideal-theoretic descriptions.  Notable examples include the set-theoretic Salmon conjecture for $\sigma_4(\mathbb{P}^3 \times \mathbb{P}^3 \times \mathbb{P}^3)$ \cite{LM04,  PS05,  AR08,  FG12},  the set-theoretic Eisenbud--Koh--Stillman conjecture for $\sigma_r(v_d(Z))$ \cite{EKS88,  Ravi94,  BGL13},  the set-theoretic Landsberg--Weyman conjecture for $\tau(\mathbb{P}^{n_1} \times \cdots \times \mathbb{P}^{n_k})$ \cite{LW07,  Oeding11}, the set-theoretic description of $v_d(Z)$ \cite{Mumford70, Griffiths83}, among yet other similar endeavors.  Here $\sigma_r(Z)$, $v_d(Z)$, and $\tau(Z)$ denote the $r$th secant,  degree-$d$ Veronese, and tangential variety of a smooth projective variety $Z$ respectively.

What we wrote in the beginning of Section~\ref{sec:deg} also applies to this section, that is, it makes no difference whether we use the projection model \eqref{eq:proj} or the involution model \eqref{eq:inv} as they only differ by a linear change of coordinates. So while Conjecture~\ref{conj:limit} was stated in \cite{KHB24} for the projection model, we may use the involution model below.

We begin by introducing some notations. Let $\overline{\Gr_\mathbb{C}(k,\mathbb{R}^n)}$ denote the \emph{projective closure} of $\Gr_\mathbb{C}(k,\mathbb{R}^n)$, i.e., its closure in the projective space $\mathbb{P}(\Sym^2(\mathbb{C}^n) \oplus \mathbb{C})$. Note that the Euclidean closure and Zariski closure are equal in this case.  The projective variety defined by homogenized generators of the ideal  of $\Gr_\mathbb{C}(k,\mathbb{R}^n)$ is  
\begin{equation}\label{eq:hm} 
\Gr^\hm_\mathbb{C}(k,\mathbb{R}^n) \coloneqq \bigl\lbrace 
[X:t] \in \mathbb{P} (\Sym^2(\mathbb{C}^n) \oplus \mathbb{C}): X^2 - t^2 I_n = 0, \; \tr(X) - (2k - n) t = 0 \bigr\rbrace.
\end{equation}  
Clearly, we have
\[
\overline{\Gr_\mathbb{C}(k,\mathbb{R}^n)} \subseteq \Gr^\hm_\mathbb{C}(k,\mathbb{R}^n)
\]
and that
\[
\partial{\Gr_\mathbb{C}(k,\mathbb{R}^n)} = \overline{\Gr_\mathbb{C}(k,\mathbb{R}^n)} \setminus \Gr_\mathbb{C}(k,\mathbb{R}^n) = 
\bigl\lbrace  [X: t] \in \overline{\Gr_\mathbb{C}(k,\mathbb{R}^n)}: t = 0 \bigr\rbrace.
\]
Let $L_{\infty}$ denote the hyperplane  at infinity, i.e., 
\begin{equation}\label{eq:L}
L_{\infty} = \bigl\lbrace 
[X:t] \in \mathbb{P} (\Sym^2(\mathbb{C}^n) \oplus \mathbb{C}): t =0 
\bigr\rbrace.
\end{equation}

\begin{lemma}\label{lem:points at infinity}
Let $S = \frac{1}{2} \begin{bsmallmatrix}
i & 1 \\
1 & -i
\end{bsmallmatrix}$. For any $[X:0] \in \Gr^\hm_\mathbb{C}(k,\mathbb{R}^n) \cap L_{\infty}$, there exist some integer $d\le \lfloor n/2 \rfloor$ and matrix $Q\in \O_n(\mathbb{C})$ such that
\[
X = Q \begin{bmatrix}
S \otimes I_d & 0 \\
0 & 0
\end{bmatrix} Q^\tp.
\]
Here $S \otimes I_d = \diag(S,\dots, S)$ is a block diagonal matrix with $d$ diagonal blocks.
\end{lemma}
\begin{proof}
Clearly $\Gr^\hm_\mathbb{C}(k,\mathbb{R}^n) \cap L_{\infty} = \{ X\in \Sym^2(\mathbb{C}^n) : X^2 = 0, \; \tr(X) = 0\}$. A complex symmetric matrix has a decomposition
\[
X = Q \diag(\lambda_1 I_{q_1} + S_1,\dots, \lambda_m I_{q_m} + S_m) Q^\tp
\]
for some $Q\in \O_n(\mathbb{C})$ and symmetric matrices $S_1,\dots,S_m$ as in Lemma~\ref{lem:canonical form}. If $X^2 = 0$, then 
\[
0 = (\lambda_j I_{q_j} + S_j)^2 = \lambda_j^2 I_{q_j} + 2 \lambda_j S_j + S_j^2,\quad j = 1,\dots, m.
\]
A direct calculation shows that $\lambda_1 = \dots = \lambda_m = 0$ and for each $j = 1,\dots, m$, we must have either
(i) $q_j = 1$ and $S_j = 0$, or (ii) $q_j = 2$ and $S_j = \frac{1}{2} \begin{bsmallmatrix}
i & 1 \\
1 & -i
\end{bsmallmatrix}$.
\end{proof}

\begin{lemma}\label{lem:components at infinity}
The variety $\Gr^\hm_\mathbb{C}(k,\mathbb{R}^n) \cap L_{\infty}$ is a union of sets $Z_1, \dots, Z_{\lfloor n/2 \rfloor}$ given by
\begin{equation}\label{eq:components at infinity}
Z_d \coloneqq \biggl\lbrace 
[X:0]\in \Gr^\hm_\mathbb{C}(k,\mathbb{R}^n):
X = Q \begin{bmatrix}
S \otimes I_d & 0 \\
0 & 0
\end{bmatrix} Q^\tp,\; Q\in \O_n(\mathbb{C})
\biggr\rbrace,\quad  d  = 1,\dots,\lfloor n/2  \rfloor.
\end{equation}
Moreover, $\dim Z_d = d(n - d)$ for each $d = 1,\dots, \lfloor n/2  \rfloor$.
\end{lemma}
\begin{proof}
From \eqref{eq:hm} and \eqref{eq:L},
\[
\Gr^\hm_\mathbb{C}(k,\mathbb{R}^n) \cap L_{\infty} = 
\bigl\lbrace  [X:0]\in \Gr^\hm_\mathbb{C}(k,\mathbb{R}^n): X^2 = 0 \bigr\rbrace.
\]
It follows from Lemma~\ref{lem:points at infinity} that we have a disjoint union (denoted by $\sqcup$ henceforth)
\[
\Gr^\hm_\mathbb{C}(k,\mathbb{R}^n) \cap L_{\infty}  = \bigsqcup_{d = 1}^{\lfloor n/2 \rfloor} Z_d
\]
where $Z_d$ is the orbit of $X_0\coloneqq \begin{bsmallmatrix}
S \otimes I_d & 0 \\
0 & 0
\end{bsmallmatrix}$ with respect to the adjoint action of $\O_n(\mathbb{C})$. Let $G_0$ be the stabilizer group of $X_0$.  Its Lie algebra $\mathfrak{g}_0$  is given by
\[
\mathfrak{g}_0 = \lbrace A \in \mathfrak{so}_n(\mathbb{C}): A X_0 = X_0 A \rbrace.
\]
We partition $A \in \mathfrak{g}_0$ as 
\[
A = \begin{bmatrix}
A_{1,1} & \cdots  & A_{1,d} & A_{1,d+1} \\
\vdots & \ddots & \vdots & \vdots \\
A_{d,1} & \cdots  & A_{d,d} & A_{d,d+1} \\
A_{d+1,1} & \cdots  & A_{d+1,d} & A_{d+1,d+1} \\
\end{bmatrix}
\]
where $A_{ij} \in \mathbb{C}^{2\times 2}$,  $A_{i,  d+1}\in \mathbb{C}^{2 \times (n-2d)}$,  $A_{d+1,  j} \in \mathbb{C}^{(n-2d) \times 2}$, $A_{d+1,d+1} \in \mathbb{C}^{(n-2d) \times (n-2d)}$,  $i, j =1,\dots, d$.  Then $A X_0 = X_0 A$ gives us
\[
A_{ij} S = S A_{ij},  \quad S A_{i,d+1} = 0,  \quad  A_{d+1,j} S = 0,  \quad i, j =1,\dots, d;
\]
and therefore
\[
\dim \mathfrak{g}_0 = 2 \binom{d}{2} + d (n-2d) + \binom{n-2d}{2} = \binom{n}{2} + d^2 - nd.
\]
It follows from the orbit-stabilizer theorem that $\dim Z_d = \dim \O_n(\mathbb{C}) - \dim G_0  = nd - d^2$.
\end{proof}

\begin{proposition}\label{prop:component}
For each $d = 1,\dots, \lfloor n/2  \rfloor$,  we have  $\overline{Z}_d = Z_1 \sqcup \dots \sqcup Z_d$. The set $Z_d$ is connected (resp.\ irreducible) unless $n$ is even and $d = n/2$,  in which case $Z_d$ has two connected (resp.\ irreducible) components.
\end{proposition}
\begin{proof} 
It follows from \eqref{eq:components at infinity} that any $[X:0] \in \overline{Z}_d$ must satisfy $X^\tp = X$,  $X^2 =0$, $\rank (X) \le d$. Thus $\overline{Z}_d \subseteq Z_1 \sqcup \dots \sqcup Z_d$ by Lemma~\ref{lem:canonical form}.  For the reverse inclusion, take any $[X: 0] \in Z_j$, $j \in \{1,\dots, d\}$, with
\[
X = Q \begin{bmatrix}
S \otimes I_j & 0 \\
0 & 0
\end{bmatrix} Q^\tp. 
\]
Define 
\[
A \coloneqq \frac{1}{2} \begin{bmatrix}
i & -1\\
-1 & -i
\end{bmatrix}, \quad B \coloneqq \frac{1}{2}\begin{bmatrix}
1 & i\\
i & -1
\end{bmatrix}, 
\]
and, for each $\varepsilon > 0$,
\[
C_\varepsilon  \coloneqq
\frac{1}{\sqrt{\varepsilon (\varepsilon + 2i)}}
\begin{bmatrix}
\varepsilon  + i & 1 \\
1 & -(\varepsilon +i)
\end{bmatrix},\quad D_\varepsilon  \coloneq 
\frac{1}{\sqrt{\varepsilon (\varepsilon + 2i)}}
\begin{bmatrix}
\varepsilon  & \sqrt{2\varepsilon i} \\
\sqrt{2\varepsilon i} & -\varepsilon 
\end{bmatrix}.
\]
Since $A$ and $B$ are symmetric matrices and $A^2 = B^2 = 0$, it follows from Lemma~\ref{lem:canonical form} that $A = Q_1 S Q_1^\tp$ and $B = Q_2 S Q_2^\tp$ for some $Q_1, Q_2 \in \O_2(\mathbb{C})$ and $S = \frac{1}{2} \begin{bsmallmatrix}
i & 1 \\
1 & -i
\end{bsmallmatrix}$.  Since $C_\varepsilon$ and  $D_\varepsilon$ are symmetric orthogonal matrices, we have
\begin{align*}
X_\varepsilon &\coloneqq Q
\begin{bmatrix}
D_\varepsilon \otimes  I_j& 0 & 0  \\
0 & C_\varepsilon \otimes I_{d-j} & 0 \\
0 & 0 & 0
\end{bmatrix}
\begin{bmatrix}
S \otimes I_d & 0 \\
0 & 0
\end{bmatrix}
\begin{bmatrix}
D_\varepsilon \otimes  I_j & 0 & 0  \\
0 & C_\varepsilon \otimes I_{d-j} & 0 \\
0 & 0 & 0
\end{bmatrix}Q^\tp \\
&= Q
\begin{bmatrix}
\frac{\varepsilon i + 2 \sqrt{2\varepsilon i} + 2}{\varepsilon + 2i} A \otimes  I_j& 0 & 0  \\
0 &   \frac{\varepsilon}{\varepsilon + 2i}  B \otimes I_{d-j} & 0 \\ 
0 & 0 & 0
\end{bmatrix}
 Q^\tp\\
&= V
\begin{bmatrix}
 \frac{\varepsilon i + 2 \sqrt{2\varepsilon i} + 2}{\varepsilon + 2i} S \otimes  I_j& 0 & 0  \\
0 &   \frac{\varepsilon}{\varepsilon + 2i}  S \otimes I_{d-j} & 0 \\ 
0 & 0 & 0
\end{bmatrix}
 V^\tp,
\end{align*}
where $V \coloneqq Q \diag(Q_1 \otimes I_j, Q_2 \otimes I_{d-j},  I_{n-d}) \in \O_n(\mathbb{C})$. Hence $[X_\varepsilon: 0] \in Z_{d}$ for any $\varepsilon > 0$ and
\[
\lim_{\varepsilon \to 0} \; [X_\varepsilon: 0]  = [X : 0] \in Z_j.
\]
The required inclusion $\overline{Z}_d \supseteq Z_1 \sqcup \dots \sqcup Z_d$ follows. Let $I_{n-1,1} \coloneqq \diag(1,\dots, 1,  -1) \in \mathbb{C}^{n\times n}$.  We have the disjoint union of cosets
\[
\O_n(\mathbb{C}) = \SO_n(\mathbb{C}) \sqcup  \bigl(
 I_{n-1,1}  \SO_n(\mathbb{C}) \bigr).
\]
This yields a decomposition $Z_d = Z^0_d \cup Z^1_d$  where  
\begin{align*}
Z^0_d &= \biggl\lbrace 
[X:0]\in \Gr^\hm_\mathbb{C}(k,\mathbb{R}^n):
X = Q \begin{bmatrix}
S \otimes I_d & 0 \\
0 & 0
\end{bmatrix} Q^\tp,\; Q\in \SO_n(\mathbb{C})
\biggr\rbrace,  \\
Z^1_d &= \biggl\lbrace 
[X:0]\in \Gr^\hm_\mathbb{C}(k,\mathbb{R}^n):
X = Q I_{n-1,1} \begin{bmatrix}
S \otimes I_d & 0 \\
0 & 0
\end{bmatrix} I_{n-1,1} Q^\tp,\; Q\in \SO_n(\mathbb{C})
\biggr\rbrace.
\end{align*}
Since $\SO_n(\mathbb{C})$ is connected, both $Z^0_d$ and $Z^1_d$ are connected.  If $1 \le d < n/2$, then
\[
I_{n-1,1} \begin{bmatrix}
S \otimes I_d & 0 \\
0 & 0
\end{bmatrix} I_{n-1,1} = \begin{bmatrix}
S \otimes I_d & 0 \\
0 & 0
\end{bmatrix};
\] 
so $Z^0_d = Z^1_d$ and $Z_d$ is connected.  If $n$ is even and $d = n/2$,  then $Z^0_d \cap  Z^1_d = \varnothing$ and so $Z_d$ has exactly two connected components. Observe also that `connected' may be replaced by `irreducible' throughout.
\end{proof}
We recall from Proposition~\ref{prop:ideal} that $\mathscr{I}_{\lfloor n/2 \rfloor,n}$ is the ideal of $\Gr_\mathbb{C}(\lfloor n/2 \rfloor,\mathbb{R}^n)$.  Since $\Gr_\mathbb{C}(\lfloor n/2 \rfloor,\mathbb{R}^n)$ is an irreducible affine variety, this ideal is radical and prime.  However, it is speculated in \cite{KHB24} that the  conjectured limit $\langle X^2,  \tr(X) \rangle$ in Conjecture~\ref{conj:limit} may fail to be prime. This follows immediately from Proposition~\ref{prop:component}: If $n$ is even, then
\[
Z \bigl( \langle X^2,  \tr(X) \rangle \bigr) = \{X \in \Sym^2(\mathbb{C}^n): X^2 = 0\} \cong \Gr^\hm_\mathbb{C}(k,\mathbb{R}^n) \cap L_{\infty} = Z_1 \sqcup \dots \sqcup Z_{\lfloor n/2 \rfloor}   = \overline{Z}_{\lfloor n/2 \rfloor}
\]
has two irreducible components.

\begin{theorem}[Projective closure of the Grassmannian]\label{thm:set}
Let $k,n \in \mathbb{Z}_\pp $ with $k\le \lfloor n/2 \rfloor$ and let 
$Z_1,\dots,Z_{\lfloor n/2 \rfloor}$ be as in \eqref{eq:components at infinity}. Then 
\begin{align*}
\overline{\Gr_\mathbb{C}(k,\mathbb{R}^n)} &= \Gr_\mathbb{C}(k,\mathbb{R}^n) \sqcup Z_1 \sqcup \dots \sqcup Z_k,  \\
\partial \Gr_\mathbb{C}(k,\mathbb{R}^n) &= Z_1 \sqcup \dots \sqcup Z_k = \overline{Z}_k.
\end{align*}
In particular,  we have 
\begin{equation}\label{eq:n/2}
\partial \Gr_\mathbb{C}(\lfloor n/2 \rfloor,\mathbb{R}^n) = \{[X:0] \in \mathbb{P}( \Sym^2(\mathbb{C}^n)\oplus \mathbb{C}) : X^2 = 0\}. 
\end{equation}
\end{theorem}
\begin{proof}
It suffices to show that $\overline{\Gr_\mathbb{C}(k,\mathbb{R}^n)} \cap L_{\infty} = Z_1 \sqcup \dots \sqcup Z_k$. By  Lemma~\ref{lem:components at infinity}, we already have
\[
\Gr^\hm_\mathbb{C}(k,\mathbb{R}^n) \cap L_{\infty} = Z_1 \sqcup \dots \sqcup Z_{\lfloor n/2  \rfloor}.
\]
Since $\overline{\Gr_\mathbb{C}(k,\mathbb{R}^n)} \subseteq \Gr^\hm_\mathbb{C}(k,\mathbb{R}^n)$, we must have $\overline{\Gr_\mathbb{C}(k,\mathbb{R}^n)} \cap L_{\infty} \subseteq Z_1 \sqcup \dots \sqcup Z_{\lfloor n/2  \rfloor}$. Since $Z_d$ is an $\O_n(\mathbb{C})$-orbit, either $Z_d\subseteq \overline{\Gr_\mathbb{C}(k,\mathbb{R}^n)} \cap L_{\infty}$ or $Z_d \cap \overline{\Gr_\mathbb{C}(k,\mathbb{R}^n)} \cap L_{\infty} = \varnothing$, for any $d = 1,\dots, \lfloor n/2  \rfloor$. By Lemma~\ref{lem:components at infinity}, $\dim Z_d = d(n-d) > k(n-k)$ if $n/2 \ge d >k$. Thus $Z_d$ is not contained in $\overline{\Gr_\mathbb{C}(k,\mathbb{R}^n)} \cap L_{\infty}$ if $n/2 \ge d >k$, from which we deduce that $\overline{\Gr_\mathbb{C}(k,\mathbb{R}^n)} \cap L_{\infty} \subseteq Z_1 \sqcup \dots \sqcup Z_k$.

For the reverse inclusion, set $X_d = \begin{bsmallmatrix}
S \otimes I_d & 0 \\
0 & 0
\end{bsmallmatrix}$ and we will construct an approximation of $[ X_d : 0 ]$ by elements in $\Gr_\mathbb{C}(k,\mathbb{R}^n) \subseteq \mathbb{P}\bigl(\Sym^2(\mathbb{C}^n) \oplus \mathbb{C}\bigr)$ for each $d = 1,\dots, k$. Let
\[
S_\varepsilon   = \frac{1}{2}
 \begin{bmatrix}
\varepsilon  + i & 1 \\
1 & -(\varepsilon +i)
\end{bmatrix},\quad T_\varepsilon  = \frac{1}{2} \begin{bmatrix}
\varepsilon  & \sqrt{2\varepsilon i} \\
\sqrt{2\varepsilon i} & -\varepsilon 
\end{bmatrix}, \quad 0 <\varepsilon < 1.
\]
Then
\[
S_\varepsilon ^2 = T_\varepsilon ^2 = \frac{\varepsilon (\varepsilon  + 2i)}{4} I_2,\quad \tr(S_\varepsilon ) = \tr(T_\varepsilon ) = 0,\quad \lim_{\varepsilon \to 0} S_\varepsilon  = S, \quad \lim_{\varepsilon \to 0} T_\varepsilon  = 0.
\]
Now let
\[
X_d(\varepsilon ) \coloneqq \begin{bmatrix}
S_\varepsilon  \otimes I_d & 0 & 0 \\
0 & T_\varepsilon  \otimes I_{k - d} \\
0 & 0 & \frac12\sqrt{\varepsilon (\varepsilon +2i)} I_{n-2k}
\end{bmatrix} \in \Sym^2(\mathbb{C}^n)
\]
so that 
\[
X_d(\varepsilon )^2 - \frac{\varepsilon (\varepsilon  + 2i)}{4} I_n = 0, \qquad  \tr(X_d(\varepsilon )) - \frac12 (n-2k) \sqrt{\varepsilon (\varepsilon +2i)} = 0.
\]
Hence we have 
\[
\biggl[ X_d(\varepsilon ): \frac{\sqrt{\varepsilon (\varepsilon +2i)}}{2} \biggr]\in \Gr_\mathbb{C}(k,\mathbb{R}^n) \subseteq \mathbb{P}(\Sym^2(\mathbb{C}^n) \oplus \mathbb{C})
\]
and
\[
\lim_{\varepsilon \to 0} \biggl[ X_d(\varepsilon ): \frac{\sqrt{\varepsilon (\varepsilon +2i)}}{2}\biggr] = [X_d:0]
\]
as required.
\end{proof}
As a consequence of Theorem~\ref{thm:set},  we have the following set-theoretic characterization of $\overline{\Gr_\mathbb{C}(k,\mathbb{R}^n)}$.
\begin{corollary}[Equations of projective closure]\label{cor:eqns for closure}
For positive integers $k \le n$, the projective closure of $\Gr_\mathbb{C}(k,\mathbb{R}^n)$ is given by
\begin{multline}\label{eq:closure}
\overline{\Gr_\mathbb{C}(k,\mathbb{R}^n)}  = 
\bigl\{ [X:t] \in \mathbb{P} (\Sym^2(\mathbb{C}^n) \oplus \mathbb{C}) : \\
 X^2 - t^2 I_n = 0, \; \rank (X + t I_n) \le k, \; \rank (X - t I_n) \le n-k \bigr\}.
\end{multline}
\end{corollary}
\begin{proof}
By Theorem~\ref{thm:set},  each $[X:t] \in \overline{\Gr_\mathbb{C}(k,\mathbb{R}^n)}$ must satisfy the three conditions in \eqref{eq:closure}.  Conversely, if $t =0$ and $[X: 0] \in \mathbb{P}(\Sym^2(\mathbb{C}^n) \oplus \mathbb{C})$ satisfies the three conditions in \eqref{eq:closure},  then setting $d \coloneqq \rank(X)  \le \min \{k,  n-k\}$ gives us  $[X:0]\in Z_d \subseteq \overline{\Gr_\mathbb{C}(k,\mathbb{R}^n)}$.  If $t \ne 0$, we may assume $t = 1$. If $[X: 1] \in \mathbb{P}(\Sym^2(\mathbb{C}^n) \oplus \mathbb{C})$ satisfies the three conditions in \eqref{eq:closure}, then the eigenvalues of $X\in \Sym^2(\mathbb{C}^n) \cap \O_n(\mathbb{C})$ must be $1$ and $-1$ with multiplicities at least $k$ and $n-k$ respectively --- but this compels them to be exactly $k$ and $n-k$. 
\end{proof}

\section{Conclusion}

We would like to add a word about the \emph{quadratic model} \cite{LK24a} of the Grassmannian, i.e.,
\[
\Gr_{a,b}(k, \mathbb{R}^n) \coloneqq \{ W \in \mathbb{R}^{n \times n}: W^\tp = W, \; (W - a I)(W- bI) = 0, \; \tr(W) = ka + (n-k)b \},
\]
where $a, b \in \mathbb{R}$ are distinct but arbitrary.
Any minimal-dimension equivariant embedding of the Grassmannian as a submanifold of matrices must take a form like this \cite{LK24a}. In particular, $(a,b)=(1,-1)$ gives the involution model used in this article and $(a,b)=(1,0)$ gives the projection model used in \cite{KHB24}.  Since any $\Gr_{a,b}(k, \mathbb{R}^n)$ is just a scaled and translated copy of $\Gr_{1,-1}(k, \mathbb{R}^n)$, the results in Sections~\ref{sec:cplx} and \ref{sec:deg} apply verbatim and the results in Section~\ref{sec:proj} apply with minor adjustments

The degree and projective closure are arguably the two most fundamental properties of an embedded variety from an algebraic geometric perspective. This article provides complete characterizations for $\Gr(k,\mathbb{R}^n)$ embedded via \eqref{eq:proj} or \eqref{eq:inv}. There is also the notion of Euclidean distance degree for an embedded variety and we have determined its value for $\Gr(k,\mathbb{R}^n)$ in \cite{ZLK24c}. Nevertheless, the Grassmannian is, in addition, a smooth manifold and \eqref{eq:proj} and \eqref{eq:inv} embed $\Gr(k,\mathbb{R}^n)$ in $\Sym^2(\mathbb{R}^n)$ or $\O_n(\mathbb{R})$. From a differential geometric perspective, these embeddings call for a study of their second fundamental forms, which may be found in \cite{ZLK24a}. 

\subsection*{Acknowledgment} The authors would like to thank the anonymous reviewer for his very helpful comments and suggestions.  LH's work is partially supported by a Vannevar Bush Faculty Fellowship ONR N000142312863. KY's work is partially supported by the National Key R\&D Program of China 2023YFA1009401 and the National Natural Science Foundation of China 12288201.

\bibliographystyle{abbrv}

\end{document}